\documentclass{amsart}
\usepackage[all]{xy}
\usepackage{amsmath,amssymb,amsthm,mathrsfs,bm}

\DeclareMathOperator*{\restprod}%
 {\mathchoice{\ooalign{\ensuremath{\displaystyle\prod}\crcr\ensuremath{\displaystyle\coprod}}}%
             {\ooalign{\ensuremath{\textstyle\prod}\crcr\ensuremath{\textstyle\coprod}}}%
             {\ooalign{\ensuremath{\scriptstyle\prod}\crcr\ensuremath{\scriptstyle\coprod}}}%
             {\ooalign{\ensuremath{\scriptscriptstyle\prod}\crcr\ensuremath{\scriptscriptstyle\coprod}}}%
 }

\theoremstyle{definition}
\newtheorem{dfn}{Definition}[section]
\newtheorem{exm}[dfn]{Example}
\newtheorem{rmk}[dfn]{Remark}
\newtheorem{pbm}[dfn]{Problem}

\theoremstyle{plain}
\newtheorem{prop}[dfn]{Proposition}
\newtheorem{lem}[dfn]{Lemma}
\newtheorem{thm}[dfn]{Theorem}
\newtheorem{cor}[dfn]{Corollary}

\def \O{\mathcal{O}}

\def \GKab{G_K^{\mathrm{ab}}}
\def \P{\mathcal{P}}
\def \p{\mathfrak{p}}
\def \m{\mathfrak{m}}
\def \a{\mathfrak{a}}
\def \b{\mathfrak{b}}

\def \Af{\mathbb{A}_{K,f}}
\def \R{\mathbb{R}}
\def \T{\mathbb{T}}
\def \Q{\mathbb{Q}}
\def \C{\mathbb{C}}
\def \Z{\mathbb{Z}}
\def \N{\mathbb{N}}
\def \H{\mathcal{H}}
\def \Prim{\mathrm{Prim}}

\title{Irreducible Representations of Bost-Connes systems}
\author{Takuya Takeishi}
\address{Department of Mathematical Sciences, University of Tokyo}
\email{takeishi@ms.u-tokyo.ac.jp}
\date{}

\begin{document}

\begin{abstract}
The classification problem of Bost-Connes systems was studied by Cornellissen and Marcolli partially, but still remains unsolved. 
In this paper, we will give a representation-theoretic approach to this problem. 
We generalize the result of Laca and Raeburn, which concerns with the primitive ideal space on the Bost-Connes system for $\mathbb{Q}$.  
As a consequence, the Bost-Connes $C^*$-algebra for a number field $K$ has $h_K^1$-dimensional irreducible representations and does not have 
finite-dimensional irreducible representations for the other dimensions, where $h_K^1$ is the narrow class number of $K$. 
In particular, the narrow class number is an invariant of Bost-Connes $C^*$-algebras. 
\end{abstract}

\maketitle

\section{Introduction}
For an arbitrary number field $K$, a $C^*$-dynamical system $(A_K,\sigma_{t,K})$ is defined 
in the work of Ha-Paugam \cite{HP}, Laca-Larsen-Neshveyev \cite{LLN} and Yalkinoglu \cite{Y}. 
The $C^*$-dynamical system $(A_K,\sigma_{t,K})$ is related to class field theory. 
It is called the {\it Bost-Connes system}, after Bost and Connes \cite{BC}, 
who defined such a system for the special case of $K=\Q$. 
It was a longstanding open problem to generalize Bost-Connes systems to arbitrary number fields, 
but that problem has been solved in recent years by the efforts of many researchers 
(especially, Yalkinoglu's work \cite{Y} was the last piece). 
So it is a good moment to start the investigation of those $C^*$-dynamical systems from both number theoretic and operator algebraic viewpoints. 
The operator algebraic viewpoint naturally asks for the classification of Bost-Connes systems. 
Concretely, we are interested in the following problem: 
\begin{pbm} \label{pbm1}
Does an $\R$-equivariant isomorphism of $(A_K,\sigma_{t,K})$ and $(A_L,\sigma_{t,L})$ imply an isomorphism of $K$ and $L$ ?
\end{pbm} 

This problem was studied by Cornelissen and Marcolli \cite{M} under the condition of preserving {\it the daggered subalgebras}, 
which has more information of number theoretic things. 
Without any additional assumption, works in the direction of the full classification tries to recover number theoretic invariants from Bost-Connes systems.  
The best known result is the classification theorem of the KMS-states by Laca-Larsen-Neshveyev \cite{LLN}, 
obtaining the Dedekind zeta function $\zeta_K(s)$ from the partition function of $(A_K,\sigma_{t,K})$. 
In particular, Problem \ref{pbm1} is true if $[K:\Q]\leq 6$ or $[L:\Q]\leq 6$, thanks to the work of R. Perlis \cite{P}. 

In this paper, we provide a new invariant of Bost-Connes systems in Theorem~\ref{classnum}, that is, the narrow class number $h_K^1$.  
The narrow class number measures the distance of the integer ring $\O_K$ from being a principal ideal domain, and some information of infinite primes is added. 
Hence, in principle, it is an independent invariant from the zeta function, which collects the information of finite primes. 
Indeed, there is an example of a pair of number fields which have the same zeta function but different narrow class numbers (Remark~\ref{example}). 
The difference between the Dedekind zeta function and the narrow class number can be viewed from an operator algebraic perspective. 
Since the flow $\sigma_{t,K}$ on $A_K$ is determined by the norms of primes, we know the information of primes by looking at flows. 
Looking at the algebra itself, we get the information which is orthogonal to finite primes. 
In particular, our theorem actually provides an invariant for $C^*$-algebras $A_K$. 

In order to prove Theorem~\ref{classnum}, we examine the primitive ideal space of $A_K$. 
There is a result of Laca and Raeburn \cite{LR} determining the primitive ideal space of the original Bost-Connes $C^*$-algebra $A_{\Q}$. 
The key ingredient in that work was Williams' Theorem~\cite{W}, 
which is a structure theorem of the primitive ideal space for group crossed products by abelian groups. 
That theorem also plays an important role in this paper.  
As a complementary result, we also determine the primitive ideal space of $A_K$ (Theorem~\ref{structure}), which is a generalization of the work of 
Laca and Raeburn. 

Looking at flows on the primitive ideal space, we get another invariant $(\hat{P}_K^1,\sigma_{t,K})$, 
which is a dynamical system on the infinite-dimensional torus (Proposition~\ref{homeo}). 
We can also recover the norm map on $P_K^1$ from that dynamical system (Theorem~\ref{isom}). 
This is a sort of results like reconstructing the norm map on the whole ideal group $J_K$, 
which amounts to the reconstruction of the zeta function by \cite{LLN}, but from a different perspective. 

\section{Preliminary} \label{Pre}

In this section, we recall the definition of Bost-Connes systems and summarize general facts and observations which are needed to investigate the primitive ideal space. For the investigation of primitive ideals, we adopt the same strategy as in the case of $\Q$ (cf. \cite{LR}).

\subsection{Definition of Bost-Connes systems} \label{def}
In this section, we quickly review the definition of the Bost-Connes system of a number field. 
The reader can also consult \cite[p.388]{Y} for the construction of the Bost-Connes system. 
Throughout this paper, $J_K$ denotes the ideal group of $K$ and $I_K$ denotes the ideal semigroup of $K$. 
The finite ad\'ele ring is denoted by $\Af$ and the finite id\'ele group is denoted by $\Af^*$ (for the definition, see e.g. \cite{LNT}). 

Let $K$ be a number field. Put 
\begin{equation*}
Y_K = \hat{\O}_K \times_{\hat{O}_K^*} \GKab,
\end{equation*}
where $\hat{\O}_K$ is the profinite completion of $\O_K$, and $\hat{\O}_K^*$ acts on $\hat{\O}_K \times \GKab$ by 
\begin{equation*}
s\cdot(\rho,\alpha) = (\rho s , [s]_K^{-1} \alpha)
\end{equation*}
for $\rho \in \hat{\O}_K,\alpha \in \GKab$ and $s \in \hat{\O}_K^*$, where $[\cdot]_K$ is the Artin reciprocity map. 
Let $\a \in I_K$ and take a finite id\'ele $a \in \Af^* \cap \hat{\O}_K$ such that $\a = (a)$. 
The action of $I_K$ on $Y_K$ is given by 
\begin{equation*}
\a \cdot [\rho,\alpha] = [\rho a, [a]_K^{-1}\alpha].
\end{equation*}

Let $A_K = C(Y_K)\rtimes I_K$. Define an $\R$-action on $A_K$ by 
\begin{equation*}
\sigma_{t,K}(f) = f,\ \sigma_{t,K}(\mu_{\a}) = N(\a)^{it}\mu_{\a}
\end{equation*}
for $f \in C(Y_K), \a \in I_K$ and $t \in \R$, where $N(\cdot)$ is the ideal norm. 
\begin{dfn}
The system $(A_K,\sigma_{t,K})$ is called the Bost-Connes system for $K$.
\end{dfn}

It is convenient to extend the Bost-Connes system to a non-unital group crossed product. 
Let 
\begin{equation*}
X_K = \Af \times_{\hat{\O}_K^*} \GKab
\end{equation*}
and define the action of $J_K$ on $X_K$ in the same way. 
Let $\tilde{A}_K = C_0(X_K)\rtimes J_K$. 
Then $A_K$ is a full corner of $\tilde{A}_K$. Namely, we have $A_K = 1_{Y_K}\tilde{A}_K1_{Y_K}$. 
The $\R$-action on $\tilde{A}_K$ is defined in the same way, which is also denoted by $\sigma_{t,K}$. 

For convenience, we fix notations of subspaces of $X_K$ and $Y_K$. 
Define four subspaces by 
\begin{eqnarray*}
&& Y_K^* = \hat{\O}_K^* \times_{\hat{\O}_K^*} \GKab \cong \GKab,\\
&& X_K^0 = \{0\} \times_{\hat{\O}_K^*} \GKab \cong \GKab/[\hat{\O}_K^*]_K, \\
&& X_K^{\natural} = (X_K \setminus \{0\}) \times_{\hat{\O}_K^*} \GKab, \\
&& Y_K^{\natural} = (Y_K \setminus \{0\}) \times_{\hat{\O}_K^*} \GKab .
\end{eqnarray*}

\subsection{Dynamics on $\hat{P}_K^1$} \label{predyn}
Since we use the dynamics on $\hat{P}_K^1$ later, we prepare it in advance. 
We fix a notation of a dynamical system on a torus. For a (finite or infinite) sequence of positive numbers $\{r_j\}$, $(\prod_j \T_j, \prod_j r_j^{it})$ 
denotes the dynamical system determined by 
\begin{equation*}
\sigma_t((x_j)_j) = (r_j^{it}x_j)_j
\end{equation*}
for $x_j \in \T$ and $t \in \R$. 

Let $K$ be a number field and $P_K^1$ denote the group of principal ideals generated by totally positive elements 
(i.e., $P_K^1 \cong K^*_+/\O_{K,+}^*$). 
We consider an action of $\R$ on $\hat{P}_K^1$ (as a topological space) defined by 
\begin{equation*}
\langle x, \sigma_t(\gamma) \rangle = N(x)^{it} \langle x,\gamma \rangle 
\end{equation*}
for any $x \in P_K^1$, $\gamma \in \hat{P}_K^1$ and $t \in \R$, where $\hat{P}_K^1$ is the Pontrjagin dual of $P_K^1$. 
Note that $P_K^1$ is a free abelian group, since it is a subgroup of the free abelian group $J_K$. 
Hence $\hat{P}_K^1$ is isomorphic to the infinite product of circles. 
If $\{a_j\}$ is a basis of $P_K^1$, then the dynamical system 
$(\hat{P}_K^1,\sigma)$ is conjugate to $(\prod_j \T_j, \prod_jN(a_j)^{it})$.

\subsection{$\R$-equivariant imprimitivity bimodules} \label{imp}
\begin{dfn}
Let $(A,\sigma_t^A)$ and $(B,\sigma_t^B)$ be $C^*$-dynamical systems. 
An $(A,B)$-imprimitivity bimodule $E$ is said to be an {\it $\R$-equivariant imprimitivity bimodule} if 
there is a one-parameter group of isometries $U_t$ on $E$ such that
\begin{itemize}
\item $_A\langle U_t \xi, U_t \eta \rangle = \sigma_t(_A\langle \xi, \eta\rangle)$
\item $\langle U_t \xi, U_t \eta \rangle_B = \sigma_t(\langle \xi, \eta\rangle_B)$
\end{itemize}
for any $\xi,\eta \in E_{\p}$ and $t \in \R$. 

If there exists an $\R$-equivariant imprimitivity bimodule, then the two $C^*$-dynamical systems are said to be 
\it{$\R$-equivariantly Morita equivalent}. 
\end{dfn}

Note that from the above axioms we have 
\begin{equation*}
\sigma_t^A (a) U_t(\xi) = U_t(a\xi),\ U_t(\xi)\sigma_t^B (b) = U_t(\xi b)
\end{equation*}
for any $a \in A, b \in B$ and $\xi \in E$. 

\begin{lem}
For a number field $K$, the Bost-Connes system $(A_K,\sigma_{t,K})$ is $\R$-equivariantly Morita equivalent to $(\tilde{A}_K, \sigma_{t,K})$.
\end{lem}
\begin{proof}
Since $A_K = 1_{Y_K}\tilde{A}_K1_{Y_K}$ and $1_{Y_K}$ is a full projection, 
the $(A_K,\tilde{A}_K)$ bimodule $E=1_{Y_K}\tilde{A}_K$ is an imprimitivity bimodule. 
Define a one-parameter group of isometries $U_t$ on $E$ by restricting the time-evolution of $\tilde{A}_K$. 
Then $U_t$ satisfies the desired property. 
\end{proof} 

If two $C^*$-algebras are Morita equivalent, then we have natural correspondences between their representations and ideals. 
As a consequence, their primitive ideal spaces are homeomorphic. The homeomorphism obtained in this way is called the {\it Rieffel homeomorphism} 
(cf.~\cite[Corollary 3.33]{RW}). 
We need an $\R$-equivariant version of this theorem. 
For a $C^*$-dynamical system $(A,\sigma_t)$, then we consider the $\R$-action on $\mathrm{Prim}A$ defined by 
\begin{equation*}
t \cdot \ker \pi = \ker (\pi \circ \sigma_t) = \sigma_{-t}(\ker \pi),
\end{equation*}
where $\pi$ is an irreducible representation of $A$. 

\begin{prop} \label{Morita}
Let $E$ be an $\R$-equivariant imprimitivity bimodule between two $C^*$-dynamical systems $(A,\sigma_t^A)$ and $(B,\sigma_t^B)$.
Then the Rieffel homeomorphism $h_X : \mathrm{Prim}B \rightarrow \mathrm{Prim}A$ is $\R$-equivariant. 
\end{prop}
\begin{proof}
Let $(\pi,\H_{\pi})$ be a representation of $B$.
We need to show that the representation $(\mathrm{id}_A \otimes 1, E \otimes_{\pi \circ \sigma_t^B} \H_{\pi})$ is unitarily equivalent to 
$(\sigma_t^A \otimes 1, E \otimes_{\pi} \H_{\pi})$. Let $U_t$ be a one-parameter group of isometries on $E$ which gives $\R$-equivariance. 
Then it is easy to check that the unitary 
\begin{equation*}
E \otimes_{\pi \circ \sigma_t^B} \H_{\pi} \rightarrow E \otimes_{\pi} \H_{\pi},\ 
x \otimes_{\pi \circ \sigma_t^B} \xi \mapsto U_t(x) \otimes_{\pi} \xi
\end{equation*}
gives the unitary equivalence. 
\end{proof}

Note that the strong continuity of the one-parameter group of isometries $U_t$ 
is tacitly assumed in the definition of $\R$-equivariant imprimitivity bimodules. 
However, the strong continuity is not needed for the sake of Proposition~\ref{Morita}. 

\subsection{The Primitive ideal space of crossed products by abelian groups} \label{primitive}
In order to determine $\Prim A_K$, by Proposition~\ref{Morita}, we may investigate $\Prim \tilde{A}_K$ instead. 
We have a nice structure theorem of the primitive ideal space for group crossed products. 
Let $G$ be a countable abelian group acting on a second countable locally compact space $X$. 
Define an equivalence relation on $X \times \hat{G}$ by 
\begin{equation*}
(x,\gamma) \sim (y,\delta) \mbox{ if } \overline{Gx} = \overline{Gy} \mbox{ and } \gamma \delta^{-1} \in G_x^{\perp},
\end{equation*}
where $\hat{G}$ is the Pontrjagin dual of $G$ and $G_x$ is the isotropy group of $x$. 
For a representation $(\pi, \H_{\pi})$ of $A_x = C_0(X) \rtimes G_x$, 
$\mathrm{Ind}^G_{G_x} \pi$ denotes the induced representation of $A = C_0(X) \rtimes G$ on the Hilbert space $A \otimes_{A_x} \H_{\pi}$. 
\begin{thm}\label{Williams} {\rm (Williams, \cite[Theorem 8.39]{W2})} 
We have a homeomorphism $\Phi: X\times \hat{G} / \sim \rightarrow \mathrm{Prim} C_0(X) \rtimes G$ defined by 
\begin{equation*}
\Phi([x,\gamma]) = \ker (\mathrm{Ind}^G_{G_x} (\mathrm{ev}_x \rtimes \gamma |_{G_x})). 
\end{equation*}
\end{thm}

\begin{rmk} \label{open}
The quotient map $X \times \hat{G} \rightarrow X \times \hat{G} / \sim$ is an open map (cf. \cite[Remark 8.40]{W2}). 
This fact is useful to determine the topology of the primitive ideal space. 
\end{rmk}

In this section, we look into the dynamics of the primitive ideal space in a general setting. 
Let $N:G\rightarrow \R_+$ be a group homomorphism and define the time evolution on $A$ by 
\begin{equation*}
\sigma_t(fu_s) = N(s)^{it} fu_s
\end{equation*}
for any $f \in C_0(X), s \in G$ and $t \in \R$. Take $x \in X, \gamma \in \hat{G}$ and let $\pi = \mathrm{ev}_x \rtimes \gamma |_{G_x}$. 
Then $\pi_x$ defines a character of $A_x$. By \cite[Proposition 8.24]{W}, $\mathrm{Ind}^G_{G_x} \pi$ is unitarily equivalent to the representation 
$\pi_{x,\gamma}$ on $\H_{x,\gamma} = C^*(G)\otimes_{C^*(G_x)} \C$ defined by 
\begin{equation*}
\pi_{x,\gamma}(f) \xi_s = f(sx) \xi_s,\ \pi_{x,\gamma}(u_t) \xi_s = \xi_{ts}
\end{equation*}
for $f \in C_0(X)$ and $s,t \in G$. The inner product of $\H_{x,\gamma}$ is defined by 
\begin{eqnarray*}
\langle \xi_s, \xi_t \rangle = 
\left\{ \begin{array}{llll} \gamma (s^{-1}t) & \mbox{if } s^{-1}t \in G_x, \\
0 & \mbox{if } s^{-1}t \not\in G_x, \end{array} \right.
\end{eqnarray*}
for any $s,t \in G$. We would like to determine the representation $\pi_{x,\gamma} \circ \sigma_t$. 
We have $\pi_{x,\gamma} \circ \sigma_t (u_s) \xi_r = N(s)^{it} \xi_{sr}$.
Let $\tilde{\H} = \H_{x,\gamma}$ as a linear space. 
Define a linear map $U:\H_{x,\gamma} \rightarrow \tilde{\H}$ by 
\begin{equation*}
U(N(s)^{it}\xi_s) = \tilde{\xi}_s
\end{equation*}
for $s \in G$. To make $U$ a unitary, the inner product on $\tilde{\H}$ needs to be defined by 
\begin{equation*}
\langle \tilde{\xi}_s, \tilde{\xi}_r \rangle 
= \left\{ \begin{array}{llll} N(s^{-1}r)^{it} \gamma (s^{-1}r) & \mbox{if } s^{-1}r \in G_x, \\
0 & \mbox{if } s^{-1}r \not\in G_x. \end{array} \right.
\end{equation*}
Then we can see that $U \pi_{x,\gamma} \circ \sigma_t U^* = \pi_{x,\tilde{\gamma}}$, where $\tilde{\gamma} = N(\cdot)^{it}\gamma$. 
Thus we have the following proposition: 
\begin{prop} \label{dyn}
Let $A= C_0(X) \rtimes G$ and consider the $\R$-action on $\mathrm{Prim}A = X \times \hat{G} / \sim$ defined in {\rm Section \ref{imp}}
 (this action is also denoted by $\sigma$). 
Then we have 
\begin{equation*}
\sigma_t([x,\gamma]) = [x,N(\cdot)^{it}\gamma] 
\end{equation*}
for $[x, \gamma] \in X \times \hat{G} / \sim$. 
\end{prop}

The Bost-Connes systems for global fields are not Type I $C^*$-algebras, because it is known that they have type III$_1$ representations. 
So we cannot expect that Williams' theorem gives complete classification of irreducible representations. 
However, we can still get some information about irreducible representations, such as their dimensions. 
We will treat that in the next section. The following lemma will be used: 
\begin{lem} \label{dim}
For $(x,\gamma) \in X \times \hat{G}$, let $(\pi_{x,\gamma}, \H_{x,\gamma})$ be the representation of $A=C_0(X) \rtimes G$ defined as above. 
Then $\dim \H_{x,\gamma} = [G:G_x]$. In particular, $\pi_{x,\gamma}$ is finite-dimensional if and only if $G_x$ has a finite index in $G$. 
\end{lem}
\begin{proof}
Let $\{s_i\}$ be a complete representative of $G/G_x$. Then the family $\{\xi_{s_i}\}$ is orthogonal in $\H_{x,\gamma}$. 
We can see that $\{\xi_{s_i}\}$ is an orthogonal basis. In fact, we have $\xi_{s_it}=\gamma(t)\xi_{s_i}$ for $t \in G_x$ because
\begin{eqnarray*}
&& \langle \gamma(t)\xi_{s_i}, \xi_{s_ir}\rangle = \gamma (t^{-1}r) = \langle \xi_{s_it}, \xi_{s_ir} \rangle, \\
&& \langle \gamma(t)\xi_{s_i}, \xi_{s_jr}\rangle = 0 = \langle \xi_{s_it}, \xi_{s_jr} \rangle,
\end{eqnarray*}
for $t,r \in G_x$ and $j \neq i$. 
\end{proof}

\begin{rmk} \label{canonical}
In fact, there is a canonical orthonormal basis of $\H_{x,\gamma}$. 
If $\{s_i\}$ is a complete set of representatives of $G/G_x$, then the family $\{ \gamma(s_i^{-1})\xi_{s_i} \}$ is an orthonormal basis and 
independent of the choice of $\{s_i\}$.
\end{rmk}

We need to study the dimensions of irreducible representations. 
Clearly, if $E$ is an $(A,B)$-imprimitivity bimodule and $\pi$ is a finite-dimensional representation of $B$, 
$E\mathrm{-Ind} \pi$ may be infinite-dimensional (e.g., $A=\mathbb{K}(\H)$ and $B=\C$). 
However, we have the following criterion in our case. 
\begin{lem} \label{crit}
Let $A$ be a $C^*$-algebra and $e \in A$ be a full projection and Let $E=eA$ be the natural $(eAe,A)$-imprimitivity bimodule. 
Let $\pi$ be a non-degenerate representation of $A$. 
Then $E\mathrm{-ind} \pi$ is unitarily equivalent to $(\pi|_{eAe}, \pi(e)\H)$. 
In particular, $\dim (E\mathrm{-ind} \pi) = \dim \pi(e)\H$. 
\end{lem}
\begin{proof}
The unitary
\begin{equation*}
eA \otimes_A \H_{\pi} \rightarrow \pi(e)\H_{\pi},\ ea\otimes \xi \mapsto \pi(ea)\xi
\end{equation*}
gives the desired unitary equivalence. 
\end{proof}

\section{Irreducible representations of Bost-Connes systems}
Hereafter, we restrict our attention to the case of Bost-Connes systems. 
We determine the structure of the primitive ideal space of $A_K$, 
investigate several examples of irreducible representations and determine the induced action of $\R$ on that space.

\subsection{Extraction of the narrow class number}
First, we prepare some arithmetic lemmas. 
For a number field $K$, $\O_{K,+}$ denotes the set of totally positive integers of $K$ 
and $U_{K,+}$ denotes the closure of $\O_{K.+}$ in $\hat{\O}_K^*$. 
The narrow ideal class group of $K$ is denoted by $C_K^1=J_K/P_K^1$. 
The following two lemmas are essentially contained in \cite[Proposition 1.1]{LNT}

\begin{lem}
The reciprocity map $[\cdot]_K:\mathbb{A}_K^*\rightarrow \GKab$ induces the isomorphism $\Af^*/\overline{K_+^*}\cong \GKab$, 
where $\overline{K_+^*}$ is the closure of $K_+^*$ in $\Af^*$. 
\end{lem}

\begin{lem} \label{fund}
The sequence
\begin{eqnarray*}
\xymatrix{
1 \ar[r] & U_K^+ \ar[r] & \hat{\O}_K^* \ar[r] & \Af^*/\overline{K_+^*} \ar[r] & C_K^1 \ar[r] & 1
}
\end{eqnarray*}
is exact. 
\end{lem}

Note that the homomorphism $\Af^*/\overline{K_+^*} \rightarrow C_K^1$ is defined by sending the class of $a \in \Af^*$ to the class of $(a)$. 
The exact sequence in Lemma~\ref{fund} plays a fundamental role in determination of the primitive ideal space. 

Combining above lemmas and Williams' theorem, we get the first main theorem. 
\begin{thm} \label{classnum}
Let $(A_K,\sigma_t)$ be the Bost-Connes system for a number field $K$ and let $h_K^1$ be the narrow class number of $K$. 
Then $A_K$ has $h_K^1$-dimensional irreducible representations, and does not have $n$-dimensional irreducible representations for $n \neq h_K^1$ and
$n < \infty$. 
\end{thm}

\begin{lem} \label{sublem}
The statement of {\rm Theorem~\ref{classnum}} holds for $\tilde{A}_K$. 
\end{lem}
\begin{proof}
Let $x= [\rho,\alpha] \in X_K = \Af \times_{\hat{\O}_K^*} \GKab$ and let $\gamma \in \hat{J}_K$. 
By Lemma~\ref{dim}, the dimension of $\pi_{x,\gamma}$ equals $[J_K: J_{K,x}]$. 
In general, if $\ker \pi = \ker \rho$ holds for irreducible representations $\pi,\rho$ of a $C^*$-algebra $A$, 
then we have $\dim \pi = \dim \rho$ because if either $\rho$ or $\pi$ is finite dimensional, 
then $A/\ker \pi \cong M_{\dim \pi}(\C)$ is isomorphic to $A/\ker \rho \cong M_{\dim \rho}(\C)$. 
Hence it suffices to show the following: 
\begin{enumerate}
\item If $\rho \neq 0$, then $[J_K: J_{K,x}] = \infty$. 
\item If $\rho = 0$, then $[J_K: J_{K,x}] = h_K^1$.
\end{enumerate}

Suppose $\rho \neq 0$ and let $\p$ be a prime of $K$ such that $\rho_{\p} \neq 0$. If $\a = (a) \in J_{K,x}$, 
then $a_{\p} \in \O_{K_{\p}}^*$ because $\rho a s = \rho$ for some $s \in \hat{O}_K^*$ implies $a_{\p}s_{\p} = 1$. 
Hence the classes of $\p^n$'s for $n \in \Z$ in $J_K/J_{K,x}$ are distinct elements. Therefore the index of $J_{K,x}$ is infinite. 

Suppose $\rho = 0$. In this case, we consider the action of $J_K$ on $X_K^0 = \GKab/[\hat{\O}_K^*]$ ($X_K^0$ is defined in Section \ref{def}). 
We have $X_K^0 = C_K^1$ by Lemma~\ref{fund}. 
The action of $J_K$ on $X_K^0 = J_K/P_K^1$ coincides with the multiplication.   
Hence the isotropy group $J_{K,x}$ coincides with $P_K^1$ and its index equals $|C_K^1|=h_K^1$. 
\end{proof}

\begin{proof}[Proof of Theorem~\ref{classnum}]
For $x=[\rho,\alpha] \in X_K$ and $\alpha \in \hat{J}_K$, let 
$(\pi_{x,\gamma}^0, \H_{x,\gamma}^0) = (\pi_{x,\gamma}|_{A_K}, \pi_{x,\gamma}(1_{Y_K})\H_{x,\gamma})$. 
We need to show that $\dim \pi_{x,\gamma} = \dim \pi_{x,\gamma}^0$. 
If $\rho = 0$, then we have $\pi_{x,\gamma}(1_{Y_K}) = 1$ by definition of $\pi_{x,\gamma}$. 
Hence $\dim \pi_{x,\gamma} = \dim \pi_{x,\gamma}^0$ holds by Lemma~\ref{crit}. 
So it suffices to show that $\pi_{x,\gamma}^0$ is infinite dimensional if $\rho \neq 0$. 

Take an integral ideal $\a \in I_K$ such that $\a x \in Y_K$ 
(we can always take such $\a$ because $\rho_{\p} \in \O_{K_{\p}}$ for all but finitely many $\p$). 
Let $\p$ be a prime of $K$ such that $\rho_{\p} \neq 0$. Then we have seen in the proof of Lemma~\ref{sublem} that 
the classes of $\p^n$'s are distinct in $J_K/J_{K,x}$. Hence so are for $\p^n\a$'s. 
This means that $\{\xi_{\p^n\a}\}_{n \in \Z}$ is an orthogonal family in $\H_{x,\gamma}$. 
Since $\p^n\a x \in Y_K$ for $n \geq 0$, $\xi_{\p^n\a} \in \pi_{a,\gamma}(1_{Y_K})\H_{x,\gamma}$ for $n \geq 0$. 
Therefore $\pi_{a,\gamma}(1_{Y_K})\H_{x,\gamma}$ is infinite dimensional. 
\end{proof}

\begin{cor}
Let $K,L$ be number fields and let $(A_K,\sigma_{t,K}), (A_L,\sigma_{t,L})$ be their Bost-Connes systems. 
If $A_K \cong A_L$ as $C^*$-algebras, then $h_K^1 = h_L^1$. 
\end{cor}

\begin{exm} \label{example}
From the classification theorem of the KMS-states by Laca-Larsen-Neshveyev \cite{LLN}, 
we know that the Dedekind zeta function is an invariant of Bost-Connes systems. 
From Theorem~\ref{classnum}, we know that the narrow class number is also an invariant. 
We can see that this is actually a new invariant. Indeed, 
there exist two fields which have the same Dedekind zeta function but different narrow class numbers. 
For example, let $K=\Q(\sqrt[8]{a}), L=\Q(\sqrt[8]{16a})$ for $a=-15$. 
Then $K$ and $L$ are totally imaginary fields, so their narrow class numbers $h_K^1,h_L^1$ are equal to their class numbers $h_K,h_L$. 
By the result of de Smit and Perlis \cite{SP}, we have $\zeta_K=\zeta_L$ and $h_K^1/h_L^1=h_K/h_L=2$. 
\end{exm}

From the proof of Theorem~\ref{classnum} and the fact $\hat{J}_K/P_K^{1,\perp} = \hat{P}_K^1$, 
we can see that there is an embedding of $\hat{P}_K^1$ into $\mathrm{Prim} A_K$. 
This is a distinguished subspace of $\mathrm{Prim} A_K$ that is homeomorphic to $\T^{\infty}$. 
By Proposition~\ref{dyn}, $\R$ acts on $\hat{P}_K^1$ as in Section \ref{predyn}. 
Hence we can get another invariant by restricting our attention to dynamics on $\hat{P}_K^1$. 

\begin{prop} \label{homeo}
Let $K,L$ be two number fields. If their Bost-Connes systems $(A_K,\sigma_{t,K})$ and $(A_L,\sigma_{t,L})$ are 
$\R$-equivariantly isomorphic, then $\hat{P}_K^1$ and $\hat{P}_L^1$ are $\R$-equivariantly homeomorphic. 
\end{prop}
\begin{proof}
Let $\Phi: \mathrm{Prim} A_K \rightarrow \mathrm{Prim} A_L$ be the $\R$-equivariant homeomorphism induced from an isomorphism between the Bost-Connes systems. 
It suffices to show that $\Phi(\hat{P}_K^1) = \hat{P}_L^1$. 
By Theorem~\ref{classnum}, $\hat{P}_K^1$ coincides with the set of all primitive ideals which have finite quotients. 
Since $\Phi$ is induced from an isomorphism, it obviously carries $\hat{P}_K^1$ to $\hat{P}_L^1$. 
\end{proof}

We study the dynamics $\hat{P}_K^1$ in Section \ref{last}. 

\subsection{Examples of Irreducible Representations}
In this section, we give an explicit description of some irreducible representations. 
As in Section \ref{primitive}, for $x \in X_K$ and $\gamma \in \hat{J}_K$ we have an irreducible representation of $\tilde{A}_K$ defined by 
\begin{equation*}
(\pi_{x,\gamma},\H_{x,\gamma}) = \mathrm{Ind}^{J_K}_{J_{K_x}} (\mathrm{ev}_x \rtimes \gamma |_{J_{K,x}}). 
\end{equation*}
By Lemma~\ref{crit}, the representation of $A_K$ corresponding to $(\pi_{x,\gamma},\H_{x,\gamma})$ is 
\begin{equation*}
(\pi_{x,\gamma}^0,\H_{x,\gamma}^0) = (\pi_{x,\gamma}|_{A_K}, \pi_{x,\gamma}(1_{Y_K})\H_{x,\gamma}). 
\end{equation*}

First, we can determine an explicit form for the finite dimensional representations. 
Since $X^0=C_K^1$ is a closed invariant set of $J_K$, we have a canonical quotient map $q_K:C(Y_K) \rtimes I_K \rightarrow C(C_K^1) \rtimes J_K$. 
Take a character $\gamma \in \hat{J}_K$. 
Then we have the $*$-homomorphism $\varphi_{\gamma}:C(C_K^1) \rtimes J_K \rightarrow C(C_K^1) \rtimes C_K^1$ defined by 
\begin{eqnarray*}
\varphi_{\gamma}(f) = f \mbox{ for } f \in C(C_K^1), \mbox{ and } \varphi_{\gamma}(u_{s}) = \langle s, \gamma \rangle u_{\bar{s}}, 
\end{eqnarray*}
where $\bar{s}$ denotes the class of $s$ in $C_K^1$. 
Since $C(C_K^1) \rtimes C_K^1 \cong M_n (\C)$ for $n = |C_K^1| = h_K^1$, 
we obtain the surjection $\varphi_{\gamma} \circ q_K : A_K \rightarrow M_n(\C)$.
As usual, the $C^*$-algebra $C(C_K^1) \rtimes C_K^1$ acts on $\ell^2(C_K^1)$ by 
\begin{eqnarray*}
(f\xi)(s)= f(s)\xi(s) \mbox{ for } f \in C(C_K^1), \mbox{ and } (u_t\xi)(s) = \xi(t^{-1}s). 
\end{eqnarray*}
So $\rho_{\gamma} = \varphi_{\gamma} \circ q_K$ defines an irreducible representation. 
If two elements $\gamma, \delta \in \hat{J}_K$ satisfy $\gamma \delta^{-1} \in \hat{P}_K^{1,\perp}$, 
then $\rho_{\gamma}$ is unitarily equivalent to $\rho_{\delta}$. 
Indeed, for any element $\omega \in \P_K^{1,\perp} \cong \hat{C}_K^1$, 
we have the isomorphism of $C(C_K^1) \rtimes C_K^1 \cong M_n(\C)$ defined by 
\begin{eqnarray*}
f \mapsto f \mbox{ for } f \in C(C_K^1), \mbox{ and } u_{\bar{s}} \mapsto \langle \bar{s}, \gamma \rangle u_{\bar{s}}, 
\end{eqnarray*}
which is automatically implemented by a unitary. 
From now on, we assume that $\rho_{\gamma}$ is associated to the element $\gamma \in \hat{J}_K/P_K^{1,\perp} \cong \hat{P}_K^1$. 

Using Remark~\ref{canonical}, we can show that $\rho_{\gamma}$ is unitarily equivalent to $\pi_{[0,1],\gamma}^0$ 
($[0,1]$ is an element of $X_K^0$, not a closed interval). 
This implies that $\{\rho_{\gamma}\}_{\gamma \in \hat{P}_K^1}$ are not mutually unitarily equivalent, 
and any finite dimensional irreducible representation is unitarily equivalent to some $\rho_{\gamma}$. 

Benefiting from writing down representations associated to $\hat{P}_K^1$ in this form, we can prove the following proposition: 
\begin{prop}
We have $\displaystyle \ker q_K = \bigcap_{\gamma \in \hat{P}_K^1} \ker \rho_{\gamma}$. 
\end{prop}
\begin{proof}
Let $A=C(C_K^1)\rtimes J_K$ and $B=C(C_K^1)\rtimes C_K^1$. 
It suffices to show the injectivity of the homomorphism $\prod \varphi_{\gamma}$. 
We distinguish $\varphi_{\gamma}$ and $\varphi_{\delta}$ for $\gamma \delta^{-1} \in P_K^{1,\perp}$ here. 
Then the range of the map 
\begin{equation*}
\prod_{\gamma \in \hat{J}_K} \varphi_{\gamma}: A \rightarrow \prod_{\gamma \in \hat{J}_K} B
\end{equation*}
is contained in $C(\hat{J}_K, B)\cong C(\hat{J}_K)\otimes B$. Let $\Phi:A\rightarrow C(\hat{J}_K)\otimes B$ be that map. 
Then we have $\Phi(fu_s)= \chi_s \otimes fu_{\bar{s}}$, where $\chi_s$ denotes the character on $\hat{J}_K$ corresponding to $s \in J_K$. 
Let $E_1:A\rightarrow C(C_K^1)$ be the canonical conditional expectation, and 
let $E_2=\mu \otimes \mathrm{id}_B:C(\hat{J}_K) \otimes B \rightarrow B$, where $\mu$ is the Haar measure of $\hat{J}_K$. 
Then $E_1$ and $E_2$ are both faithful conditional expectations, and the diagram
\begin{eqnarray*}
\xymatrix{
A \ar[r]^{\Phi} \ar[d]^{E_2} & C(\hat{J}_K) \otimes B \ar[d]^{E_2} \\
C(C_K^1) \ar[r] & B=C(C_K^1) \rtimes C_K^1
}
\end{eqnarray*}
commutes. This implies the injectivity of $\Phi$. 
\end{proof}

\begin{cor}
Let $K,L$ be number fields. Then any isomorphism from $A_K$ to $A_L$ carries 
$\ker q_K=C_0(Y_K^{\natural}) \rtimes I_K$ to $\ker q_L=C_0(Y_L^{\natural}) \rtimes I_L$. 
\end{cor}

Next, we visit another example. 
By the KMS-classification theorem in \cite{LLN}, extremal KMS$_{\beta}$-states for $\beta > 1$ are obtained from irreducible representations. 
Let us recall the definition of these representations. 
For $g \in \GKab$, we have an irreducible representation $\pi_g$ on $\ell^2(I_K)$ defined by 
\begin{eqnarray*}
&&\pi_g(f)\xi_s = f(s \cdot g)\xi_s \mbox{ for } f \in C(Y_K), \mbox{ and } \\
&&\pi_g(\mu_t)\xi_s = \xi_{ts} \mbox{ for } t \in I_K, 
\end{eqnarray*}
where $g$ is identified with $[1,g] \in Y_K^*$. 
We can check that $\pi_g$ is unitarily equivalent to $\pi_{g,1}^0$ because $\pi_{g,1}(1_{Y_K})$ coincides with the projection 
$\ell^2(J_K)\rightarrow \ell^2(I_K)$. 

We can see directly that these representations are not unitarily equivalent. 

\begin{prop}
The representations $\{\pi_g\}_{g}$ are not unitarily equivalent. 
\end{prop}
\begin{proof}
We have the tensor product decomposition of the Hilbert space as follows: 
\begin{equation*}
\ell^2(I_K) \cong \bigotimes_{\p} \ell^2(\N_{\p}),\ 
\xi_{\prod_{\p \in F} \p^{k_{\p}}} \mapsto \bigotimes_{\p \in F} \xi_{k_{\p}} \otimes \bigotimes_{\p \not\in F} 1,
\end{equation*}
where $\N_{\p}$ is a copy of $\N$ and $F$ is a finite set of primes of $K$. 
In this decomposition, the $C^*$-subalgebra $C^*(I_K)$ of $\mathbb{B}(\ell^2(I_K))$ moves to $\bigotimes_{\p} T_{\p}$, 
where $T_{\p}$ is a copy of the Toeplitz algebra ($T_{\p}$ is generated by the unilateral shift on $\ell^2(\N_{\p})$). 
Since $T_{\p}$ contains $\mathbb{K}(\ell^2(\N_{\p}))$, its commutant is trivial. 
Hence the commutant of $C^*(I_K)$ is trivial. 

Suppose that $\pi_g$ and $\pi_h$ are unitarily equivalent. 
Then the implementing unitary $U$ commutes with $C^*(I_K)$. 
The above argument implies $U=1$, so we have $\pi_g = \pi_h$. Hence $g=h$. 
\end{proof}

We would like to see where these representations are located inside $\mathrm{Prim} A_K$. 
Note that if $x \in Y_K^*$ then $J_{K,x}$ is trivial. 
So we have to determine $\overline{J_Kx}$ for $x \in Y_K$.  

\begin{lem} \label{Kaction} {\rm (cf.~ \cite[Lemma 2.3]{LR})}
For $\rho \in \Af$, we have 
\begin{equation*}
\overline{K^*_+ \rho} = \{ \sigma \in \Af\ |\ \rho_{\p}=0 \mbox{ {\rm implies} } \sigma_{\p}=0 \}.
\end{equation*}
\end{lem}
\begin{proof}
We may assume $\rho \in \hat{\O}_K$ because $\overline{K^*_+ a \rho} = \overline{K^*_+ \rho}$ for any $a \in \O_{K,+}$
and the right hand side is invariant under multiplication by an element of $\Af^*$. 
Take $\sigma$ from the right hand side. 
Enumerate the primes of $K$ as $\p_1,\p_2,\dots $. 
Define $\tau \in \Af$ by 
\begin{eqnarray*}
\tau_{\p} = \left\{ \begin{array}{llll} 
\rho_{\p}^{-1}\sigma_{\p} & \mbox{if } \rho_{\p} \neq 0, \\
0                         & \mbox{if } \rho_{\p} = 0. \end{array} \right. 
\end{eqnarray*}
Take $a \in \O_{K,+}$ satisfying $a\tau \in \hat{\O}_K$. 
For each $n$, take $k_n \in \O_{K,+}$ such that $k_n \equiv a\tau_{\p} \mbox{ mod } \p^n$ for $\p=\p_k$ with $1\leq k\leq n$. 
Then we have $a\sigma \in \hat{\O}_K$ and $k_n \rho_{\p} \equiv a\sigma_{\p} \mbox{ mod } \p^n$ for such $\p$. 
This implies that $k_n \rho$ converges to $a \sigma$ in $\Af$, so $a^{-1}k_n \rho$ converges to $\sigma$. 
The other inclusion is obvious. 
\end{proof}
\begin{lem} \label{Iaction}
For $x=[\rho,\alpha] \in X_K$, we have 
\begin{equation*}
\overline{J_K x} = \{ y=[\sigma, \beta] \in \Af\ |\ \rho_{\p}=0 \mbox{ {\rm implies} } \sigma_{\p}=0 \}.
\end{equation*}
\end{lem}
\begin{proof}
Take $y=[\sigma,\beta]$ from the right hand side. 
Take a finite id\'ele $a \in \Af^*$ such that $\alpha [a]_K^{-1} = \beta$ and let $\a$ be the ideal generated by $a$. 
Then $\a [\rho,\alpha] = [\rho a, \beta]$. 
By Lemma~\ref{Kaction}, there exists a sequence $k_n \in K^*_+$ such that $k_n \rho a$ converges to $\sigma$. 
Since $[k_n]_K=1$, the sequence $(k_n)\a x$ converges to $y$. 
\end{proof}

As a conclusion, $\pi_g$'s have the same kernel although they are not unitarily equivalent. 
Indeed, by Theorem~\ref{Williams}, $\ker \pi_g = \ker \pi_h$ if and only if $\overline{J_K g} = \overline{J_K h}$. 
The condition $\overline{J_K g} = \overline{J_K h}$ is true for any $g,h$ by Lemma~\ref{Iaction}. 

In fact, we have the following proposition: 
\begin{prop} {\rm (cf.~\cite[Proposition 2.10]{LR})}
The representations $\pi_g$'s are faithful. 
\end{prop}
\begin{proof}
It suffices to see that the conditional expectation $E:C(Y_K) \rtimes I_K \rightarrow C(Y_K)$ is recovered by $\pi_g$. 
From Lemma~\ref{Iaction}, we have $\overline{I_K g} = Y_K$. 
Indeed, if the sequence $\a_n g$ for $\a_n \in J_K$ converges to some $x \in Y_K$, 
then $\a_n g \in Y_K$ for large $n$, which implies $\a_n \in I_K$ for large $n$. 
Hence $C(Y_K)$ can be embedded into $\prod_{\a \in I_K} \C$ by $f\mapsto \prod_{\a \in I_K} f(\a g)$. 
For $\a \in I_K$, let $\varphi_{\a}$ be the vector state $\langle \cdot \xi_{\a}, \xi_{\a} \rangle$ on $\mathbb{B}(\ell^2(I_K))$. 
Define a unital completely positive map $E'$ by 
\begin{equation*}
E' = \prod_{\a \in I_K} \varphi_{\a} : \mathbb{B}(\ell^2(I_K)) \rightarrow \prod_{\a \in I_K} \C. 
\end{equation*}
Then $E = E' \circ \pi_g$, which completes the proof. 
\end{proof}


\subsection{The formal description of the primitive ideal space} 
The purpose of this section is to study the equivalence relation that appeared in Section \ref{primitive} in detail. 
So this section amounts to an actual generalization of the work of Laca and Raeburn \cite{LR}. 
We have already studied quasi-orbits of $J_K$ in Lemma~\ref{Iaction}, so it suffices to see what the isotropy group is. 
Let $K$ be a number field. The symbol $\P_K$ denotes the set of all finite primes of $K$. 
For a finite subset $S$ of $\P_K$, define the subgroup $\Gamma_S$ of $J_K$ by 
\begin{equation*}
\Gamma_S = \{(a)\ |\ a \in \overline{K^*_+} \subset \Af^*, a_{\p}=1 \mbox{ for } \p \not\in S \}. 
\end{equation*} 
Note that $\Gamma_S$ is a subgroup of $P_K^1$, because $\overline{K^*_+}$ is contained in $K^*_+\hat{O}^*_K$. 
We can see that $\Gamma_{\emptyset} = 1$ and $\Gamma_{\P_K} = P_K^1$. 

For $x=[\rho,\alpha] \in X_K$, let $S_x=\{ \p \in \P_K\ |\ \rho_{\p}=0 \}$. 
By Lemma~\ref{Iaction}, for $x,y \in X_K$, $\overline{J_Kx} = \overline{J_Ky}$ if and only if $S_x=S_y$. 

\begin{lem} \label{isotropy} {\rm (cf.~\cite[Lemma 2.1]{LR})}
For $x \in X_K$, the isotropy group $J_{K,x}$ coincides with $\Gamma_{S_x}$. 
\end{lem}
\begin{proof}
Let $\a \in J_{K,x}$. Take $\rho \in \Af$ and $\alpha \in \GKab$ such that $x=[\rho,\alpha]$. 
Then we can choose a finite id\'ele $a \in \Af$ generating $\a$ and satisfies $[a]_K=1$ and $\rho a = \rho$. 
Hence $a$ belongs to $\overline{K^*_+}$ and $a_{\p} = 1$ for $\p$ satisfying $\rho_{\p} \neq 0$. 
This implies that $\a \in \Gamma_{S_x}$. The converse inclusion can be shown in a similar way. 
\end{proof}

Combining Lemma~\ref{Iaction}, Lemma~\ref{isotropy} and Theorem~\ref{Williams}, we get the following conclusion. 
\begin{thm} \label{structure}
We have $\displaystyle \mathrm{Prim} A_K = \bigcup_{S \subset \P} \hat{\Gamma}_S$, where $S$ runs through all subsets of $\P$. 
\end{thm}

Theorem~\ref{structure} does not say anything about the topology of $\mathrm{Prim} A_K$. 
Actually, the only important fact is that the inclusion $\hat{\Gamma}_S \hookrightarrow \mathrm{Prim}A_K$ is a homeomorphism onto its range. 
However, we describe the topology of $\mathrm{Prim} A_K$ explicitly for the sake of completeness. 

\begin{dfn} (cf.~\cite[pp.437]{LR})
Let $2^{\P}$ be the power set of $\P$. 
The {\it power-cofinite topology} of $2^{\P}$ is the topology generated by 
\begin{equation*}
U_F = \{ S \in 2^{\P}\ |\ S \cap F = \emptyset \}, 
\end{equation*}
where $F$ is a finite subset of $\P$. 
\end{dfn}
Note that $\{U_F\}_F$ is a basis of the topology since we have $U_{F_1} \cap U_{F_2} = U_{F_1 \cup F_2}$. 

\begin{prop} {\rm (cf.~\cite[Proposition 2.4]{LR})}
The canonical surjection 
\begin{equation*}
Q:2^{\P} \times \hat{J}_K \rightarrow \bigcup_{S \subset \P} \hat{\Gamma}_S = \mathrm{Prim} A_K,
\ (S, \gamma) \mapsto \gamma|_{\Gamma_S} \in \hat{\Gamma}_S
\end{equation*}
is an open continuous surjection. 
\end{prop}
\begin{proof}
Define $Q_1: X_K \times \hat{J}_K \rightarrow 2^{\P} \times \hat{J}_K$ by sending $(x,\gamma)$ to $(S_x,\gamma)$. 
Let $Q_2:X_K \times \hat{J}_K\rightarrow \mathrm{Prim} A_K = X_K \times \hat{J}_K / \sim$ be the natural quotient map. 
Then we have $Q_2=Q \circ Q_1$. 
The quotient map $\Af \times \GKab \rightarrow \Af \times_{\hat{O}_K^*} \GKab = X_K$ is denoted by $R$. 
Then we can show in the same way as in \cite[Proposition 2.4]{LR} that  
\begin{eqnarray*}
Q_1\left( R\left( \prod_{\p \in F} V_{\p} \times \prod_{\p \not\in F} \O_{K,\p} \times V \right) \times W \right) 
= U_G \times W,\ \mbox{and} \\
Q_1^{-1}(U_F \times W) = R\left( \prod_{\p \in F} K_{\p}^* \times \restprod_{\p \not\in F} (K_{\p},\hat{O}_{K_{\p}}) \times \GKab \right) \times W
\end{eqnarray*}
for a finite set $F$ of $\P$, non-empty open sets $V_{\p}$ of $K_{\p}$, $V$ of $\GKab$ and $W$ of $\hat{J}_K$, 
where $G=\{\p \in F\ |\ 0 \not\in V_{\p}\}$. 
This means that $Q_1$ is open and continuous. Since $Q_1$ is surjective and $Q_2=Q \circ Q_1$ is open and continuous by Remark~\ref{open}, 
$Q$ is also an open and continuous surjection. 
\end{proof}

Let us briefly view when two points in $\mathrm{Prim} A_K$ can be separated by open sets. Take two distinct subsets $S_1, S_2$ of $\P$. 
If $S_1 \not\subset S_2$, 
then $Q(U_G \times \hat{J}_K) \cap \hat{\Gamma}_{S_1} = \emptyset$ and $Q(U_G \times \hat{J}_K) \supset \hat{\Gamma}_{S_2}$ 
for any finite subset $G$ of $S_1 \setminus S_2$. 
Hence, if $S_1 \not\subset S_2$ and $S_2 \not\subset S_1$, 
then $\hat{\Gamma}_{S_1} \cup \hat{\Gamma}_{S_2}$ is Hausdorff with respect to the relative topology. 
If $S_1 \subset S_2$, then any open set which contains $\hat{\Gamma}_{S_2}$ also contains $\hat{\Gamma}_{S_1}$. 

We can say that $\Prim A_K$ is a bundle over $2^{\P}$ with fibers $\hat{\Gamma}_S$. 
In other words, $\Prim A_K$ is considered as a net of compact groups indexed by subsets of $\P$. 
It seems difficult to determine the group $\Gamma_S$ in general. 
However, if $K=\Q$ or $K$ is imaginary quadratic, then $\Gamma_S$ is trivial for $S \neq \P$ because $K^*_+$ is closed in $\Af^*$. 
In such cases, we have
\begin{equation*}
\mathrm{Prim} A_K = 2^{\P}\setminus \{\P\} \cup \hat{P}_K^1. 
\end{equation*}

\begin{prop}
Let $K,L$ be imaginary quadratic fields. Then any $\R$-equivariant homeomorphism $\Prim A_K\rightarrow \Prim A_L$ induces 
an $\R$-equivariant homeomorphism $\hat{P}_K^1 \rightarrow \hat{P}_L^1$. 
In particular, if $A_K$ and $A_L$ are $\R$-equivariantly Morita equivalent, then the conclusion of {\rm Proposition~\ref{homeo}} is true. 
\end{prop}
\begin{proof}
Let $\Phi: \mathrm{Prim} A_K \rightarrow \mathrm{Prim} A_L$ be an $\R$-equivariant homeomorphism. 
It suffices to show that $\Phi(\hat{P}_K)=\hat{P}_L$. 
By Proposition~\ref{dyn}, $\R$ acts on $2^{\P}\setminus \{\P\}$ trivially and acts on $\hat{P}_K$ as in Section \ref{Pre}. 
Let $\gamma \in \hat{P}_K$ and suppose $\Phi(\gamma) \not\in \hat{P}_L$. 
Then we have $\Phi(\gamma)=x$ for some $x \in 2^{\P}\setminus \{\P\}$. Since $\Phi$ is $\R$-equivariant, we have $\Phi(\R \cdot \gamma)=x$. 
However, the orbit of $\gamma$ is clearly an infinite set, which is a contradiction. 
Therefore $\Phi(\gamma) \in \hat{P}_L$, so we have $\Phi(\hat{P}_K) \subset \hat{P}_L$. 
Hence, by symmetry, we have $\Phi(\hat{P}_K)=\hat{P}_L$. 
\end{proof}

\section{The Dynamics of $\hat{P}_K^1$} \label{last}
In this section, we prove the second main theorem. 
\begin{thm} \label{isom}
Let $K,L$ be number fields. If their Bost-Connes systems $(A_K,\sigma_{t,K})$ and $(A_L,\sigma_{t,L})$ are 
$\R$-equivariantly isomorphic, then we have a group isomorphism $P_K^1\rightarrow P_L^1$ which preserves the norm map. 
\end{thm}

Since we have Proposition~\ref{homeo}, the above theorem is reduced to the following proposition: 

\begin{prop} \label{deform}
Let $K,L$ be number fields. If $\hat{P}_K^1$ and $\hat{P}^1_L$ are $\R$-equivariantly homeomorphic, 
then there exists an $\R$-equivariant isomorphism between them. 
\end{prop}

\begin{rmk}
If $\hat{\varphi}: \hat{P}_L^1 \rightarrow \hat{P}_K^1$ is an $\R$-equivariant isomorphism, then the isomorphism $\varphi:P_K^1 \rightarrow P_L^1$ 
induced by $\hat{\varphi}$ preserves the norm. 
Indeed, let $a \in P_K^1$ and $b=\varphi(a) \in P_L^1$. Then, by taking the Pontrjagin duals, we have the following commutative diagram: 
\begin{eqnarray*}
\xymatrix{
\hat{P}_L^1 \ar[r]^{\sim }_{\hat{\varphi}} \ar[d] & \hat{P}_K^1 \ar[d] \\
(\hat{b^{\Z}},N(b)^{it}) \ar[r]^{\sim } & (\hat{a^{\Z}},N(a)^{it}). 
}
\end{eqnarray*}
The isomorphism $\hat{\varphi}$ is $\R$-equivariant by assumption, and it is easy to show that the vertical maps are $\R$-equivariant. 
Using these facts, we can show that the isomorphism $\hat{b^{\Z}}\rightarrow \hat{a^{\Z}}$ is $\R$-equivariant. 
This implies that $N(a)=N(b)$. 
\end{rmk}

Note that the isomorphism in Proposition~\ref{deform} is not canonical. 
The key observation is that the space $\hat{P}_K^1$ has a nice orbit decomposition. 

\begin{lem} \label{decomp}
Let $K$ be a number field. 
The compact group $\hat{P}^1_K$ is $\R$-equivariantly isomorphic to $(\prod_{j=1}^{\infty} \T_j \times \T^{\infty}, \prod_{j=1}^{\infty} n_j^{it} \times 1)$, 
where $n_j > 1$ and $\{n_j\}$ is linearly independent over $\Z$ in the free abelian group $\Q_+^*$. 
\end{lem}
\begin{proof}
Let $N:P_K^1 \rightarrow \Q_+^*$ be the ideal norm and let $A=N(P_K^1)$. 
Then the exact sequence
\begin{eqnarray*}
\xymatrix {
0 \ar[r] & \ker N \ar[r] & P_K^1 \ar[r]^{N} & A \ar[r] & 0 }
\end{eqnarray*}
splits, because $\ker N, P_K^1$ and $A$ are all free abelian groups. 
Let $s:A\rightarrow P_K^1$ be the splitting of $N$, and take a basis $\{a_j\}_j$ of $s(A)$. 
Then we have the decomposition 
\begin{equation*}
P_K^1 = \bigoplus_j a_j^{\Z} \oplus \ker N. 
\end{equation*}
Taking the Pontrjagin duals, we have the desired decomposition. 
\end{proof}

\begin{rmk} \label{minimal}
The condition that $\{n_j\}$ is linearly independent in $\Q_+^*$ means that 
the homeomorphism on $\prod_j \T_j$ by multiplying $\prod_j n_j^{it}$ is minimal for appropriate $t \in \R$. 
Indeed, the family $\{ 1, \frac{t}{2\pi} \log n_j \}$ is linearly independent over $\Q$ if we choose $t=2\pi$. 
\end{rmk}

\begin{proof}[Proof of Proposition~\ref{deform}]
Let $\varphi: \hat{P}^1_K \rightarrow \hat{P}^1_L$ be an $\R$-equivariant homeomorphism. 
Take the decomposition
\begin{eqnarray*}
&&P_K^1 = \bigoplus a_j^{\Z} \oplus \ker N_K, \hat{P}^1_K = (\prod_j \T_j \times \T^{\infty}, \prod_j N(a_j)^{it} \times 1), \\
&&P_L^1 = \bigoplus b_k^{\Z} \oplus \ker N_L, \hat{P}^1_L = (\prod_k \T_k \times \T^{\infty}, \prod_k N(b_k)^{it} \times 1)
\end{eqnarray*}
as in Lemma~\ref{decomp}. 
By Remark~\ref{minimal}, We have the closed orbit decomposition
\begin{eqnarray*}
\hat{P}^1_K = \coprod_{x \in \T^{\infty}}\prod_j \T_j \times \{x\}, 
\hat{P}^1_L = \coprod_{y \in \T^{\infty}}\prod_k \T_k \times \{y\}. 
\end{eqnarray*}
Hence we have $\varphi(\prod_j \T_j \times \{1\}) = \prod_k \T_k \times \{y\}$ for some $y \in  \T^{\infty}$, 
so $\varphi$ induces an $\R$-equivariant homeomorphism 
\begin{equation*}
\bar{\varphi} : (\prod_j \T_j, \prod_j N(a_j)^{it}) \rightarrow (\prod_k \T_k, \prod_k N(b_k)^{it}). 
\end{equation*}
Let $\psi = \bar{\varphi}(1)^{-1}\bar{\varphi}$ and $x = \prod_j N(a_j)^{2\pi i}, y = \prod_k N(b_k)^{2\pi i}$. 
Then we have $\psi(a^l) = b^l$ for any $l \in \Z$. 
Hence $\psi$ is an $\R$-equivariant group isomorphism, since $a$ and $b$ generates dense subgroups in $\prod_j \T_j$ and $\prod_k \T_k$ respectively. 
Taking any group isomorphism $\tau$ of $\T^{\infty}$, we obtain an $\R$-equivariant group isomorphism $\psi \times \tau: \hat{P}^1_K\rightarrow \hat{P}^1_L$. 
\end{proof}

\begin{rmk}
By the classification theorem of the KMS-states in \cite{LLN}, we know that if the Bost-Connes systems of two number fields $K,L$ are isomorphic then 
their Dedekind zeta functions are the same, which implies that there exists a group isomorphism $J_K\rightarrow J_L$ which preserves the norm. 
\end{rmk}
By Theorem~\ref{isom}, the pair $(P_K^1, N:P_K^1\rightarrow \Q^*_+)$ is an invariant of Bost-Connes systems. 
The difference between $(P_K^1, N:P_K^1\rightarrow \Q^*_+)$ and $(J_K, N:J_K\rightarrow \Q^*_+)$ is thought to be very subtle 
because $P_K^1$ is of finite index in $J_K$. We do not know what difference exists between the two invariants. 
Instead, we can see that large information which is obtained by $(J_K, N:J_K\rightarrow \Q^*_+)$ 
can also be obtained by $(P_K^1, N:P_K^1\rightarrow \Q^*_+)$. Here is an example: 
\begin{prop}
Let $K,L$ be number fields with $n=[K:\Q]=[L:\Q]$. 
Suppose that there exists a group isomorphism $P_K^1\rightarrow P_L^1$ which preserves the norm. 
Then for rational prime $p$, $p$ is non-split in $K$ if and only if $p$ is non-split in $L$. 
\end{prop}
\begin{proof}
It suffices to show the equivalence of the following conditions: 
\begin{enumerate}
\item $p$ is non-split in $K$.
\item There does not exist an element $a$ in $K^*_+$ satisfying $1 \leq v_p(N(a)) < n$, where $v_p$ denotes the valuation of $\Q$ at $p$. 
\end{enumerate}

Suppose that $p$ is non-split in $K$. Then any element $a \in K^*_+$ satisfying $1 \leq v_p(N(a))$ is a multiple of $p$ in $K$. 
Hence $n \leq v_p(N(a))$ holds for such $a$. 

Suppose that $p$ splits in $K$ and let $(p)= \prod \p_i^{e_i}$ be the prime decomposition of $p$. 
Put $\p=\p_1$. By assumption, we have $1\leq v_p(N(\p)) < n$. 
Let $\m=\prod \p_i$ and let $J_K^{\m}/P_K^{\m}$ be the ray class group modulo $\m$. 
Since the natural map $J_K^{\m}/P_K^{\m}\rightarrow J_K/P_K^1$ is surjective, 
we can choose a fractional ideal $\b$ that is prime to $(p)$ and satisfies $\b \p \in P_K^1$. 
Then $a= \b \p$ satisfies $1 \leq v_p(N(a)) < n$. 
\end{proof}

\begin{exm}
Two quadratic fields $K,L$ can be distinguished by primes which are non-split in $K$ and $L$, 
because non-splitness of primes can be known by the Legendre symbol (cf.~\cite[Chapter I, Proposition 8.5]{N}, \cite[Chapter VI, Proposition 14]{S}). 
Hence, all Bost-Connes systems for quadratic fields are mutually non-isomorphic. 
This fact can also be obtained by the KMS classification theorem. 
So Theorem~\ref{isom} gives another proof of this fact. 
\end{exm}

\section*{Acknowledgments}
This work was supported by the Program for Leading Graduate Schools, MEXT, Japan and JSPS KAKENHI Grant Number 13J01197. 
The author would like to thank to Gunther Cornelissen, Yasuyuki Kawahigashi, Marcelo Laca and Makoto Yamashita for fruitful conversations. 

\bibliographystyle{amsplain} 
 
\bibliography{ref}

\end{document}